\documentclass[12pt]{article}
\usepackage{amsmath}
\usepackage{amsthm}
\usepackage{amsfonts}
\usepackage{eucal}
\theoremstyle{plain}
\newtheorem{lem}{Lemma}
\newtheorem{thm}{Theorem}
\newtheorem{coroll}{Corollary}

\theoremstyle{definition}
\newtheorem{remark}{Remark}
\newtheorem*{proof1}{Proof of Theorem \ref{poly}}

\def\be{\beta}

\def\zt{\zeta}
\def\si{\sigma}
\def\zts{\zeta^{\star}}
\def\R{\mathbf R}

\newcommand{\N}{\ensuremath{\mathbb{N}}}
\newcommand{\E}{\ensuremath{\mathbb{E}}}
\newcommand{\fallfak}[2]{\ensuremath{#1^{\underline{#2}}}}
\newcommand{\stirone}[2]{\genfrac{ [ }{ ] }{0pt}{}{#1}{#2}}
\newcommand{\stirtwo}[2]{\genfrac{ \{ }{ \} }{0pt}{}{#1}{#2}}

\DeclareMathOperator{\Unif}{\text{Uniform}}
\newcommand{\bo}[1]{\ensuremath{\mathbf{#1}}}

\begin{document}
\title{An Asymptotic Series for an Integral}
\author{Michael E. Hoffman, Markus Kuba, Moti Levy, and Guy Louchard}
\date{March 10, 2018}
\maketitle
\begin{abstract}
We obtain an asymptotic series $\sum_{j=0}^\infty\frac{I_j}{n^j}$ for the 
integral $\int_0^1[x^n+(1-x)^n]^{\frac1{n}}dx$ as $n\to\infty$, and compute 
$I_j$ in terms of alternating (or ``colored'') multiple zeta value.
We also show that $I_j$ is a rational polynomial the ordinary zeta values,
and give explicit formulas for $j\le 12$.
As a byproduct, we obtain precise results about the convergence of norms of random variables and their moments. 
We study $\Vert(U,1-U)\Vert_n$ as $n$ tends to infinity and we also discuss $\Vert(U_1,U_2,\dots,U_r)\Vert_n$ for standard uniformly distributed random variables.
\end{abstract}
\section{Introduction}
\par\noindent
Let
\begin{equation}
\label{idef}
I(n)=\int_0^1[x^n+(1-x)^n]^{\frac1{n}}dx .
\end{equation}
We shall obtain an asymptotic series
\[
I(n)=I_0+\frac{I_1}{n}+\frac{I_2}{n^2}+\frac{I_3}{n^3}+\cdots
\]
This integral has been discussed in~\cite{L} (together with a different problem proposed by M.D. Ward). Therein, it as treated by a different approach using Euler sums and polylogarithms, leading to the first few terms $I_0$ up $I_7$ in terms of multiple zeta values.

\smallskip

Here, we give a complete expansion of $I(n)$. 
The coefficients $I_k$ can be written in terms of alternating or 
``colored'' multiple zeta values. 
The multiple zeta values are defined by
\[
\zt(i_1,\dots,i_k)=\sum_{n_1>\cdots>n_k\ge 1}\frac1{n_1^{i_1}\cdots n_k^{i_k}}
\]
for positive integers $i_1,\dots,i_k$ with $i_1>1$.
This notation can be extended to alternating or ``colored'' multiple
zeta values by putting a bar over those exponents with an associated
sign in the numerator, as in
\[
\zt(\bar 3,\bar1,1)=\sum_{n_1>n_2>n_3\ge 1}\frac{(-1)^{n_1+n_2}}{n_1^3n_2n_3} .
\]
Note that $\zt(a_1,a_2,\dots,a_k)$ converges unless $a_1$ is an 
unbarred 1.  We have $\zt(\bar1)=-\log 2$ and
\[
\zt(\bar n)=(2^{1-n}-1)\zt(n)
\]
for $n\ge 2$.  Alternating multiple zeta values have been extensively
studied, and some identities for them are established in \cite{BBB}.
Our formula for $I_k$, $k\ge 2$, can be stated as
\begin{equation}
\label{altf}
I_k=\frac{(-1)^k}2\sum_{j=2}^kE_{2\lfloor\frac{j-1}2\rfloor+1}(0)
\zt(\bar j,\underbrace{1,\dots,1}_{k-j}) ,
\end{equation}
where $E_n$ is the $n$th Euler polynomial.
But in fact the right-hand side of Eq. (\ref{altf}) can always be
rewritten as a rational polynomial in the ordinary zeta values $\zt(i)$, 
$i\ge 2$.  This follows from an identity
of K\"olbig \cite{K} that relates alternating multiple zeta values 
$\zt(\bar n,1,\dots,1)$ and multiple zeta values $\zt(n,1,\dots,1)$.
\smallskip 

After our main result, we interpret the integral $I(n)$ as the expected value of a certain random variable $Z_n$, defined in terms of the $n$th norm of the random vector $(U,1-U)$. Here, $U$ denotes a standard uniformly distributed random variable.
We complement our analysis of $I(n)=\E(Z_n)$ by studying the positive real moments $\E(Z_n^s)$ in terms of (alternating) multiple zeta values, as $n$ tends to infinity. Moreover, we also discuss as a counterpart the $n$th norm of the random vector $(U_1,U_2,\dots,U_r)$ for $r\ge 2$ and derive its moments in terms of multiple zeta values and related sums. 

\section{Main result: a complete expansion of I(n)}
Because of the symmetry around $x=\frac12$ in (\ref{idef}), one can
write
\[
I(n)=2\int_0^{\frac12}[x^n+(1-x)^n]^{\frac1n}dx=2\int_0^{\frac12}(1-x)
\left[1+\left(\frac{x}{1-x}\right)^n\right]^{\frac1{n}}dx .
\]
Now let $u=\frac{x}{1-x}$, or $x=\frac{u}{1+u}$.  Then
$dx=\frac{du}{(1+u)^2}$, and we have
\[
I(n)=2\int_0^1\left(1-\frac{u}{1+u}\right)(1+u^n)^{\frac1{n}}\frac{du}{(1+u)^2}=
2\int_0^1(1+u^n)^{\frac1{n}}\frac{du}{(1+u)^3} .
\]
Writing $(1+u^n)^{\frac1{n}}$ as $\exp\left(\frac1{n}\log(1+u^n)\right)$
and expanding the exponential in series, we have
\[
I(n)=2\int_0^1\left(1+\sum_{k=1}^\infty\frac1{k!}\left(\frac1{n}\log(1+u^n)\right)^k
\right)\frac{du}{(1+u)^3} .
\]
Now we can write (see \cite[p. 351]{GKP})
\begin{equation}
\label{logp}
\left(\log(1+x)\right)^k=k!\sum_{m=1}^\infty \frac{x^m}{m!}s(m,k),
\end{equation}
where the $s(m,k)$ are (signed) Stirling numbers of the first kind.
Hence
\begin{multline*}
I(n)=2\int_0^1\frac{du}{(1+u)^3}+2\sum_{k=1}^\infty\int_0^1k!\sum_{m=1}^\infty
\frac{u^{mn}s(m,k)}{m!n^kk!}\frac{du}{(1+u)^3}\\
=\frac34+2\sum_{k=1}^\infty\frac1{n^k}
\sum_{m=1}^\infty\frac{s(m,k)}{m!}\int_0^1\frac{u^{mn}}{(1+u)^3}du .
\end{multline*}
If we let $\zt_r(i_1,\dots,i_k)$ denote the truncated multiple zeta value
\[
\zt_r(i_1,\dots,i_k)=\sum_{r\ge n_1>n_2>\dots>n_k\ge 1}
\frac1{n_1^{i_1}n_2^{i_2}\cdots n_k^{i_k}} ,
\]
then we have the following relation, which is well-known although
perhaps not in this notation (cf. \cite{A}).
\begin{lem}
\label{stirmzv}
For positive integers $m\ge k$,
\[
s(m,k)=(-1)^{m-k}(m-1)!\zt_{m-1}(\{1\}_{k-1}),
\]
where $\{1\}_m$ means 1 repeated $m$ times.
\end{lem}
\begin{proof}
From the relation
\[
x(x-1)\cdots (x-n+1)=\sum_{n=0}^n s(n,k)x^k
\]
it follows that $s(n,k)=(-1)^{n-k}e_{n-k}(1,2,\dots,n-1)$, were $e_j$
is the $j$th elementary symmetric function. Divide by $(n-1)!$ to get
\[
\frac{s(n,k)}{(n-1)!}=(-1)^{n-k}\frac{e_{n-k}(1,2,\dots,n-1)}{(n-1)!}
=(-1)^{n-k}e_{k-1}\left(1,\frac12,\dots,\frac1{n-1}\right),
\]
and the conclusion follows since evidently $\zt_{n-1}(\{1\}_{n-k})=
e_{k-1}\left(1,\frac12,\dots,\frac1{n-1}\right)$.
\end{proof}
Thus
\begin{equation}
\label{Ieq}
I(n)=\frac34+2\sum_{k=1}^\infty \frac1{n^k}\sum_{m=1}^\infty(-1)^{m-k}
\frac{\zt_{m-1}(\{1\}_{k-1})}{m}\int_0^1\frac{u^{mn}}{(1+u)^3}du .
\end{equation}
\par
If we write
\[
\int_0^1 \frac{u^r}{(1+u)^3}du=\sum_{j=1}^\infty \frac{\be_{j-1}}{r^j} ,
\]
then the $\be_j$ can be computed explicitly as follows.
\begin{lem}
\[
\be_j=\frac{(-1)^j}4(E_{j+1}(-1)+E_{j+2}(-1)),
\]
where the $E_j$ are Euler polynomials.  
\end{lem}
\begin{proof}
Making the change of variable $u=e^{-t}$, we have
\[
\int_0^1\frac{u^r}{(1+u)^3}du=\int_0^\infty \frac{e^{-t}}{(1+e^{-t})^3}e^{-rt}dt .
\]
By direct computation
\[
\frac{e^{-t}}{(1+e^{-t})^3}=\frac14\left[\frac{d^2}{dt^2}\left(\frac{2e^t}
{1+e^{-t}}\right)-\frac{d}{dt}\left(\frac{2e^t}{1+e^{-t}}\right)\right].
\]
The generating function of the Euler polynomials is defined by 
\begin{equation}
\label{EulerGF}
\mathcal{E}(t,x)=\frac{2e^{tx}}{1+e^t}=\sum_{j\ge 0}E_j(x)\frac{t^j}{j!}.
\end{equation}
Differentiating $\mathcal{E}(-t,-1)$ gives
\[
\frac{d}{dt}\left(\frac{2e^t}{1+e^{-t}}\right)=
-\sum_{n=0}^\infty (-1)^n E_{n+1}(-1)\frac{t^n}{n!}
\]
and
\[
\frac{d^2}{dt^2}\left(\frac{2e^t}{1+e^{-t}}\right)=
\sum_{n=0}^\infty (-1)^n E_{n+2}(-1)\frac{t^n}{n!} .
\]
Hence
\begin{multline*}
\int_0^\infty\frac{e^{-t}}{(1+e^{-t})^3}e^{-rt}dt
=\sum_{n=0}^\infty \frac{(-1)^n}4(E_{n+2}(-1)+E_{n+1}(-1))\int_0^\infty 
\frac{t^n}{n!}e^{-rt}dt\\
=\sum_{n=0}^\infty \frac{(-1)^n}4(E_{n+2}(-1)+E_{n+1}(-1))\frac1{r^{n+1}} ,
\end{multline*}
from which the conclusion follows.
\end{proof}
The well-known identity
\begin{equation}
\label{eulerid}
E_n(x)+E_n(x+1)=2x^n
\end{equation}
gives $E_n(-1)=2(-1)^n-E_n(0)$,
so that 
\[
E_{j+1}(-1)+E_{j+2}(-1)=-E_{j+1}(0)-E_{j+2}(0) .
\]
But $E_n(0)=0$ for $n$ even, so we have
\[
\be_j=\begin{cases} \frac14E_{j+2}(0),&\text{if $j$ is odd,}\\
-\frac14E_{j+1}(0),&\text{if $j$ is even,}\end{cases}
\]
or more succinctly $\be_j=(-1)^{j+1}\frac14E_{2\lfloor\frac{j+1}2\rfloor+1}(0)$.
If we set $a_n=\frac12E_{2n+1}(0)$, then 
$2(-1)^{j-1}\be_j=a_{\lfloor\frac{j+1}2\rfloor}$.
The $a_n$ can be written in terms of Bernoulli numbers as
\[
a_n=\frac{(1-2^{2n+2})B_{2n+2}}{2n+2},
\]
and we also have the exponential generating function
\[
\sum_{n=0}^\infty a_n\frac{t^{2n+1}}{(2n+1)!}=-\frac12\tanh\frac{t}2 .
\]
The first few $a_j$ are
\[
a_0=-\frac14,\ 
a_1=\frac18,\
a_2=-\frac14,\
a_3=\frac{17}{16},\
a_4=-\frac{31}4,\
a_5=\frac{691}8,\
a_6=-\frac{5461}4 .
\]
\par
Using Eq. (\ref{Ieq}) we can write
\begin{align*}
I(n)&=\frac34+2\sum_{k=1}^\infty \sum_{m=1}^\infty \sum_{j=1}^\infty
(-1)^{m-k}\frac{\zt_{m-1}(\{1\}_{k-1})\be_{j-1}}{n^{j+k}m^{j+1}}\\
&=\frac34+2\sum_{p=2}^\infty \frac1{n^p}\sum_{m=1}^\infty\sum_{k=1}^{p-1}
(-1)^{m-k}\frac{\zt_{m-1}(\{1\}_{k-1})\be_{p-k-1}}{m^{p-k+1}}\\
&=\frac34+2\sum_{p=2}^\infty\frac1{n^p}\sum_{k=1}^{p-1}(-1)^k\be_{p-k-1}
\zt(\overline{p-k+1},\{1\}_{k-1}) ,
\end{align*}
from which we see that $I_0=\frac34$, $I_1=0$, and 
\[
I_p=2\sum_{k=1}^{p-1}(-1)^k\be_{p-k-1}\zt(\overline{p-k+1},\{1\}_{k-1})
=2\sum_{j=2}^p(-1)^{p-j-1}\be_{j-2}\zt(\bar j,\{1\}_{p-j}) 
\]
for $p\ge 2$.  We have proved the following result.
\begin{thm}
\label{main}
For $p\ge 2$,
\[
I_p=(-1)^p\sum_{j=2}^p a_{\lfloor\frac{j-1}2\rfloor}\zt(\bar j,\{1\}_{p-j}),
\]
where $a_n=\frac12E_{2n+1}(0)=(1-2^{2n+2})B_{2n+2}/(2n+2)$.
\end{thm}
The first two cases are as follows.
\begin{align*}
I_2&=a_0\zt(\bar2)=\frac18\zt(2)\\
I_3&=-a_0\zt(\bar2,1)-a_1\zt(\bar3)=\frac14\cdot\frac{\zt(3)}8
+\frac18\cdot\frac34\zt(3)=\frac18\zt(3) .
\end{align*}
In all further computations, expressions for alternating multiple zeta values
are simplified using the Multiple Zeta Value Data Mine \cite{BBV}.
By Theorem \ref{main}, 
\[
I_4=a_0\zt(\bar2,1,1)+a_1\zt(\bar3,1)+a_2\zt(\bar4)=
-\frac14\zt(\bar2,1,1)+\frac18\zt(\bar3,1)+\frac18\zt(\bar4),
\]
and since $\zt(\bar4)=-\frac78\zt(4)$, $\zt(\bar2,1,1)=-\frac1{16}\zt(4)
+\frac12\zt(\bar3,1)$, this implies $I_4=-\frac{3}{32}\zt(4)$.
Similarly, 
\begin{multline*}
I_5=-a_0\zt(\bar2,1,1,1)-a_1\zt(\bar3,1,1)-a_1\zt(\bar4,1)-a_2\zt(\bar5)=\\
\frac14\zt(\bar2,1,1,1)-\frac18\zt(\bar3,1,1)-\frac18\zt(\bar4,1)+\frac14
\zt(\bar5) .
\end{multline*}
Now $\zt(\bar5)=-\frac{15}{16}\zt(5)$, and from \cite{BBV}
\begin{align*}
\zt(\bar4,1)&=-\frac{29}{32}\zt(5)+\frac12\zt(2)\zt(3)\\
\zt(\bar2,\{1\}_3)&=\frac{31}{64}\zt(5)-\frac14\zt(2)\zt(3)
+\frac12\zt(\bar3,1,1),
\end{align*}
giving the result $I_5=-\frac18\zt(2)\zt(3)$.
\par
Here, without further details, are $I_j$ for $j=6,7,8,9,10,11,12$.
\begin{align*}
I_6&=\frac{83}{256}\zt(6)-\frac1{16}\zt(3)^2\\
I_7&=\frac3{16}\zt(7)+\frac{27}{64}\zt(3)\zt(4)+\frac3{16}\zt(2)\zt(5)\\
I_8&=-\frac{2533}{1536}\zt(8)+\frac3{16}\zt(3)\zt(5)
+\frac5{32}\zt(2)\zt(3)^2\\
I_9&=-\frac56\zt(9)-\frac{289}{128}\zt(3)\zt(6)-\frac{135}{64}
\zt(4)\zt(5)-\frac98\zt(2)\zt(7)+\frac5{96}\zt(3)^3\\
I_{10}&=\frac{293937}{20480}\zt(10)-\frac{87}{32}\zt(3)\zt(7)-\frac9{16}\zt(5)^2
-\frac{81}{64}\zt(3)^2\zt(4)-\frac{21}{16}\zt(2)\zt(3)\zt(5)\\
I_{11}&=\frac{63}8\zt(11)+\frac{58007}{3072}\zt(3)\zt(8)
+\frac{5187}{256}\zt(5)\zt(6)+\frac{135}8\zt(4)\zt(7)
+\frac{115}{12}\zt(2)\zt(9)\\
&-\frac{13}{48}\zt(2)\zt(3)^3-\frac{21}{32}\zt(3)^2\zt(5)\\
I_{12}&=-\frac{2095281645}{11321344}\zt(12)+\frac{115}{12}\zt(3)\zt(9)
+\frac{81}8\zt(5)\zt(7)+\frac{5765}{512}\zt(3)^2\zt(6)\\
&+\frac{1323}{64}\zt(3)\zt(4)\zt(5)+\frac{45}4\zt(2)\zt(3)\zt(7)
+\frac{45}8\zt(2)\zt(5)^2-\frac{13}{192}\zt(3)^4
\end{align*}
\par
In fact, the $I_n$ are always rational polynomials in the ordinary
zeta values $\zt(i)$, in consequence of the following result.
\begin{thm}
\label{poly}
For $p\ge 2$,
\[
I_p=\frac{(-1)^p}2\sum_{k=1}^{p-1}(-1)^k\zt(k+1,\{1\}_{p-k-1})\sum_{j=0}^{k-1}
\binom{k-1}{j}a_{\lfloor \frac{p-1-j}2\rfloor} .
\]
\end{thm}
The proof makes use of an identity of K\"olbig \cite{K}, which is phrased
in terms of the integral
\[
S_{n,p}(z)=
\frac{(-1)^{n+p-1}}{(n-1)!p!}\int_0^1\frac{\log^{n-1}(t)\log^p(1-zt)}{t}dt.
\]
But $S_{n,p}(z)$ can be written as a multiple zeta value if $z=1$,
and as an alternating multiple zeta value if $z=-1$.  The key is the 
following result.
\begin{lem}
If $|z|\le 1$, then
\[
\frac{(-1)^{n+p-1}}{(n-1)!p!}\int_0^1\frac{\log^{n-1}(t)\log^p(1-zt)}{t}dt
=\sum_{j_1>j_2>\cdots>j_p\ge 1}\frac{z^{j_1}}{j_1^{n+1}j_2\cdots j_p} .
\]
\end{lem}
\begin{proof}
Since
\[
\log(1-zt)=-\sum_{i\ge 1}\frac{z^it^i}{i}\quad\text{and}\quad
\int_0^1t^{m-1}\log^{n-1}(t)dt=\frac{(n-1)!}{m^n},
\]
we have
\begin{multline*}
\int_0^1\frac{\log^{n-1}(t)\log^p(1-zt)}{t}dt\\
=(-1)^p\sum_{i_1=1}^\infty\sum_{i_2=1}^\infty\cdots\sum_{i_p=1}^\infty
\int_0^1\frac{z^{i_1+\dots+i_p}t^{i_1+\dots+i_p-1}
\log^{n-1}(t)}{i_1i_2\cdots i_p}dt\\
=(-1)^p\sum_{i_1=1}^\infty\sum_{i_2=1}^\infty\cdots\sum_{i_p=1}^\infty
\frac{(-1)^{n-1}(n-1)!z^{i_1+\dots+i_p}}{i_1i_2\cdots i_p(i_1+\dots+i_p)^n} .
\end{multline*}
By \cite[Lemma 4.3]{H92}, this is 
\[
(-1)^p\sum_{j_1>j_2>\dots>j_p\ge 1}\frac{(-1)^{n-1}(n-1)!p!z^{j_1}}{j_1^{n+1}j_2\cdots j_p}
\]
and the conclusion follows.
\end{proof}
It then follows from definitions that
\[
S_{n,p}(1)=\zt(n+1,\{1\}_{p-1})\quad\text{and}\quad
S_{n,p}(-1)=\zt(\overline{n+1},\{1\}_{p-1}) .
\]
In \cite{K} K\"olbig refers to $S_{n,p}(1)$ as $s_{n,p}$ and $S_{n,p}(-1)$
as $(-1)^p\si_{n,p}$; the result we need is \cite[Theorem 3]{K}, which reads
\begin{equation}
\label{kint}
\sum_{j=1}^n\binom{n+p-j-1}{p-1}\si_{j,n+p-j}
+\sum_{j=1}^p\binom{n+p-j-1}{n-1}\si_{j,n+p-j}=s_{n,p} .
\end{equation}
\begin{proof1}
Note that we can rewrite Theorem \ref{main} as 
\[
I_p=\sum_{i=1}^{p-1}(-1)^ia_{\lfloor\frac{i}2\rfloor}\si_{i,p-i}
\]
and Eq. (\ref{kint}) as
\[
\sum_{i=1}^{p-1}\left(\binom{p-i-1}{p-j-1}+\binom{p-i-1}{j-1}\right)
\si_{i,p-i} = s_{j,p-j} .
\]
If we can find $\rho_j$ so that
\begin{multline*}
\sum_{j=1}^{p-1}\rho_js_{j,p-j}
=\sum_{j=1}^{p-1}\sum_{i=1}^{p-1}
\rho_j\left(\binom{p-i-1}{p-j-1}+\binom{p-i-1}{j-1}\right)\si_{i,p-i}\\
=\sum_{i=1}^{p-1}\si_{i,p-i}\sum_{j=1}^{p-1}
\rho_j\left(\binom{p-i-1}{p-j-1}+\binom{p-i-1}{j-1}\right)
=I_p,
\end{multline*}
i.e.,
\begin{equation}
\label{system}
\sum_{j=1}^{p-1}\rho_j\left(\binom{p-i-1}{p-j-1}+\binom{p-i-1}{j-1}\right)
=(-1)^ia_{\lfloor\frac{i}2\rfloor}
\end{equation}
for $i=1,2,\dots,p-1$, then $I_p$ can be written in terms of the
$s_{m,n}$.
Now Eqs. (\ref{system}) can be written
\[
\sum_{j=1}^{p-i}\rho_{p-j}\binom{p-i-1}{j-1}+
\sum_{j=1}^{p-i}\rho_j\binom{p-i-1}{j-1}=(-1)^ia_{\lfloor\frac{i}2\rfloor} ,\
1\le i\le p-1,
\]
and if we make the condition $\rho_{p-j}=\rho_j$, this becomes
\begin{equation}
\label{symsys}
\sum_{j=1}^{p-i}\rho_j\binom{p-i-1}{j-1}=\frac{(-1)^i}2a_{\lfloor\frac{i}2
\rfloor} ,\ 1\le i\le p-1,
\end{equation}
Restrict the system (\ref{symsys}) to the last $\lfloor\frac{p}2\rfloor$ 
equations ($i=\lfloor\frac{p+1}2\rfloor,\dots,p-1$) and use binomial 
inversion to get
\begin{equation}
\label{soln}
\rho_k=\frac{(-1)^{p+k}}{2}\sum_{j=0}^{k-1}a_{\lfloor\frac{p-j-1}2 \rfloor}
\binom{k-1}{j},\ 1\le k\le \lfloor\tfrac{p}2\rfloor .
\end{equation}
We claim that $\rho_k$ so defined, if the definition is extended to 
$1\le k\le p-1$, is also a solution of the first $\lfloor\frac{p-1}2\rfloor$
equations of (\ref{symsys}).  The conclusion then follows.
\par
To prove the claim, it is enough to show that the extension of 
Eqn. (\ref{soln}) to $1\le k\le p-1$ is consistent with the condition 
$\rho_{p-k}=\rho_k$, i.e., that
\[
\frac{(-1)^k}{2}\sum_{j=0}^{p-k-1}a_{\lfloor\frac{p-j-1}2\rfloor}\binom{p-k-1}{j}=
\frac{(-1)^{p+k}}{2}\sum_{j=0}^{k-1}a_{\lfloor\frac{p-j-1}2\rfloor}\binom{k-1}{j},
\]
or, using the definition of $a_n$,
\[
\sum_{j=0}^{p-k-1}E_{2\lfloor\frac{p-j-1}2\rfloor+1}(0)\binom{p-k-1}{j}=
(-1)^p\sum_{j=0}^{k-1}E_{2\lfloor\frac{p-j-1}2\rfloor+1}(0)\binom{k-1}{j} .
\]
By considering the cases $p$ odd and $p$ even, we see this can be written
\[
\sum_{j=0}^{p-k}E_{p-j}(0)\binom{p-k}{j}=(-1)^p\sum_{j=0}^k E_{p-j}(0)\binom{k}{j}.
\]
The result then follows from taking $n=p-k$ in Lemma \ref{eulerz} below.
\end{proof1}
\begin{lem}
\label{eulerz}
For nonnegative integers $n,k$,
\[
\sum_{j=0}^nE_{k+j}(0)\binom{n}{j}=(-1)^{n+k}\sum_{j=0}^kE_{n+j}(0)\binom{k}{j}
\]
\end{lem}
\begin{proof}
Start with 
\[
\sum_{j=0}^nE_j(0)\binom{n}{j}=-E_n(0)
\]
which follows from setting $x=0$ in the identity (\ref{eulerid}).
Since $E_n(0)=0$ for $n$ even, we can write this as 
\[
\sum_{j=0}^nE_j(0)\binom{n}{j}=(-1)^nE_n(0), 
\]
which is the case $k=0$ of the conclusion.  We can then use it as
the base case of a proof of the conclusion by induction on $k$.
We have
\begin{multline*}
(-1)^{n+k+1}\sum_{j=0}^{k+1}E_{n+j}(0)\binom{k+1}{j}=\\
(-1)^{n+k+1}\left[\sum_{j=1}^{k+1}E_{n+j}(0)\binom{k}{j-1}
+\sum_{j=0}^kE_{n+j}(0)\binom{k}{j}\right]=\\
(-1)^{n+k+1}\left[\sum_{j=0}^kE_{n+1+j}(0)\binom{k}{j}
+\sum_{j=0}^kE_{n+j}(0)\binom{k}{j}\right]=\\
\sum_{j=0}^{n+1}E_{k+j}(0)\binom{n+1}{j}
-\sum_{j=0}^nE_{k+j}(0)\binom{n}{j}=
\sum_{j=1}^{n+1}E_{k+j}(0)\binom{n}{j-1}\\
=\sum_{j=0}^nE_{k+1+j}(0)\binom{n}{j}.
\end{multline*}
\end{proof}
\begin{coroll}
For $p\ge 2$, $I_p$ is a rational polynomial in the the $\zt(i)$.
\end{coroll}
\begin{proof}
For any positive integers $n,m$ the multiple zeta value $\zt(n+1,\{1\}_m)$ is 
a rational polynomial in the $\zt(i)$, as follows from \cite[Eq. (10)]{BBB}.
Then Theorem \ref{poly} implies the conclusion.
\end{proof}
\section{Applications: convergence of norms}
Let $U=\Unif[0,1]$ denote a standard uniformly distributed random variable.
Furthermore, for positive real $n$ we define random variables $Z_n$ by
\[
Z_n=\Vert(U,1-U)\Vert_n=\big(U^n+(1-U)^n\big)^{\frac1n}.
\]
From the theory of norms we expect that the limit $Z_\infty$ exists and
\[
Z_\infty=\Vert(U,1-U)\Vert_\infty=\max\{U,1-U\}.
\]
It is known that $\max\{U,1-U\}=\Unif[\frac12,1]$. It turns out that our 
previous considerations allow to refine this intuition.
The integral $I(n)$ treated in detail before is exactly the expected value 
of $Z_n$. 
In the following we give asymptotic expansion of all positive real moments 
of $Z_n$. 

\begin{thm}
\label{TheoProb1}
The random variable $Z_n$, defined in terms of $U=\Unif[0,1]$, converges for $n\to\infty$ in distribution and with convergence of all integer moments,
\[
Z_n=\big(U^n+(1-U)^n\big)^{\frac1n}\to Z_\infty=\max\{U,1-U\},
\]
For positive integer $s\ge 1$ we have 
\begin{multline*}
\E(Z_n^s)=\frac{2(1-\frac{1}{2^{s+1}})}{s+1}+\\
\sum_{p=2}^\infty\frac{(-1)^p}{n^p}\sum_{k=1}^{p-1}\frac{s^{k}}{s+1}\sum_{j=1}^{s+1}\gamma_{s+1,j}(-1)^{j-1} E_{p-k+j-1}(0)\zt(\overline{p+1-k},\{1\}_{k-1}),
\end{multline*}
where the values $\gamma_{s+1,j}$ are given by $\frac{(-1)^{j-1}\stirone{s+1}{j}}{s!}=(-1)^{j-1}\zt_s(\{1\}_{j-1})$.

\smallskip

For arbitrary positive real $s>0$ we have
\begin{align*}
\E(Z_n^s)&=\frac{2(1-\frac{1}{2^{s+1}})}{s+1}+\sum_{p=2}^\infty \frac{(-1)^p}{n^p}\sum_{k=1}^{p-1}\frac{s^k}{s+1}\zt(\overline{p+1-k},\{1\}_{k-1})\\
&\quad\times \sum_{\ell=1}^{p-k}\fallfak{(s+1)}{\ell}B_{p-k,\ell}(E_1(0),\dots,E_{p-k-\ell+1}(0)),
\end{align*}
where $B_{n,k}(x_1,\dots,x_{n+1-k})$ denote the Bell polynomials.
\end{thm}

A first by product of our moment expansions is a rate of convergence. 

\begin{coroll}
The distribution functions $F_n(x)=\mathbb{P}\{Z_n\le x\}$ and $F_{\infty}(x)=\mathbb{P}\{Z_\infty\le x\}$ satisfy
\label{CorollProb1}
\[
\sup_{x\in \R}| F_{n}(x)-F_{\infty}(x)| \le \frac{C}{n}. 
\]
\end{coroll}

We also can directly strengthen to almost-sure convergence. 

\begin{coroll}
\label{CorollProb2}
The random variable $Z_n=\big(U^n+(1-U)^n\big)^{\frac1n}$ converges almost surely to $Z_{\infty}=\max\{U,1-U\}$.
\end{coroll}

\begin{remark}
We obtain in a similar way moment convergence of random variables
\[
Z_n=\big(B^n+(1-B)^n\big)^{\frac1n},
\]
with $B$ denoting a $Beta(\alpha,\beta)$ distributed random variable with real $\alpha,\beta>0$,
generalizing our results above (case $\alpha=\beta=1$). 
\end{remark}

We note that
\[
\mathbb{E}(Z_n^s)=\int_{\Omega}\Big(\big(U^n+(1-U)^n\big)^{\frac1n}\Big)^s d\mathbb{P}=\int_{0}^{1}\Big(x^n+(1-x)^n\Big)^{\frac{s}n}dx.
\]
Proceeding as before we use the symmetry of the integrand. 
\[
\mathbb{E}(Z_n^s)=2\int_0^{\frac12}(1-x)^s
\left[1+\left(\frac{x}{1-x}\right)^n\right]^{\frac{s}{n}}dx.
\]
Substituting again $u=\frac{x}{1-x}$, or $x=\frac{u}{1+u}$, leads to
\[
\mathbb{E}(Z_n^s)=2\int_0^1\left(1-\frac{u}{1+u}\right)^s(1+u^n)^{\frac{s}{n}}\frac{du}{(1+u)^2}=
2\int_0^1(1+u^n)^{\frac{s}{n}}\frac{du}{(1+u)^3} .
\]
Writing $(1+u^n)^{\frac{s}{n}}$ as $\exp\left(\frac{s}{n}\log(1+u^n)\right)$
and expanding the exponential in series, we have
\[
\mathbb{E}(Z_n^s)=2\int_0^1\left(1+\sum_{k=1}^\infty\frac1{k!}\left(\frac{s}{n}\log(1+u^n)\right)^k
\right)\frac{du}{(1+u)^{s+2}}.
\]
As before,
\begin{multline*}
\mathbb{E}(Z_n^s)=2\int_0^1\frac{du}{(1+u)^{s+2}}
+2\sum_{k=1}^\infty\int_0^1k!\sum_{m=1}^\infty
\frac{u^{mn}s(m,k) s^k}{m!n^k k!}\frac{du}{(1+u)^{s+2}}\\
=\frac{2(1-\frac{1}{2^{s+1}})}{s+1}+\sum_{k=1}^\infty \frac{s^k}{n^k}
\sum_{m=1}^\infty(-1)^{m-k}
\frac{\zt_{m-1}(\{1\}_{k-1})}{m}\int_0^1\frac{2u^{mn}}{(1+u)^{s+2}}du.
\end{multline*}
It remains to expand the integral into powers of $n$. 
Make the substitution $u=e^{-t}$ and then integrate by parts:
\begin{multline*}
\int_0^1\frac{2u^{mn}}{(1+u)^{s+2}}du
=\int_0^\infty \frac{2e^{-t}}{(1+e^{-t})^{s+2}}e^{-nmt}dt=\\
-\frac{1}{2^s(s+1)} + \frac{nm}{s+1}\int_0^\infty \frac{2}{(1+e^{-t})^{s+1}}e^{-mnt}dt.
\end{multline*}

We adapt the previous result for $s=1$ using derivative polynomials.
Changing the sign of the variable $t$ in~\eqref{EulerGF} and evaluation at $x=0$ gives
\[
\mathcal{E}(-t,0)=\frac{2}{1+e^{-t}}=\sum_{j\ge 0}(-1)^j E_j(0)\frac{t^j}{j!}.
\]
Thus, for our base function we choose the logistic function 
\[
f(t)=\frac12\mathcal{E}(-t,0) =\frac1{1+e^{-t}}.
\]
\begin{lem}[Derivative polynomials - logistic function]
For positive integer $r$ the derivative $f_r(z):=\frac{d^{r-1}}{dt^{r-1}}f(t)$ 
can be written as a polynomial in $f$:
\[
f_r(z)=\sum_{j=1}^{r}c_{r,j}\cdot f(t)^j=\sum_{j=1}^{r}\frac{c_{r,j}}{(1+e^{-t})^j}.
\]
The numbers $c_{r,j}$ are explicitly given by
\[
(-1)^{j-1}(j-1)!\stirtwo{r}{j} ,
\]
where $\stirtwo{n}{k}$ is the number of ways to partition $\{1,2,\dots,n\}$
into $k$ nonempty subsets (Stirling number of the second kind).
In particular, $c_{r,1}=1$ and $c_{r,r}=(r-1)!(-1)^{r-1}$.
\end{lem}
\begin{proof}
In \cite{H95} a general theory of derivative polynomials is developed:  if
$f$ is a function such that $f'(t)=P(f(t))$ for a polynomial function $P$,
then evidently $f^{(n)}(t)=P_n(f(t))$ for polynomials $P_n$, and if we let
\[
F(x,t)=\sum_{n\ge 0}\frac{t^n}{n!}P_n(x)
\]
then \cite[Theorem 1]{H95} gives
\begin{equation}
\label{gfform}
F(x,t)=f(f^{-1}(x)+t) .
\end{equation}
In the case $f(t)=(1+e^{-t})^{-1}$, Eq. (\ref{gfform}) gives
\[
\sum_{n\ge 0}\frac{t^n}{n!}P_n(x)=\frac{x}{x+(1-x)e^{-t}}=\frac{xe^t}{1+x(e^t-1)}
=xe^t\sum_{m=0}^\infty (-1)^mx^m(e^t-1)^m .
\]
Using the identity
\[
(e^t-1)^m=m!\sum_{p\ge m}\stirtwo{p}{m}\frac{t^p}{p!},
\]
this becomes
\begin{multline*}
\sum_{n\ge 0}\frac{t^n}{n!}P_n(x)=xe^t\sum_{m=0}^\infty(-1)^mx^mm!\sum_{p\ge m}
\stirtwo{p}{m}\frac{t^p}{p!}=\\
\sum_{q=0}^\infty\frac{t^q}{q!}\sum_{p=0}^\infty\frac{t^p}{p!}\sum_{m=0}^p
(-1)^mx^{m+1}m!\stirtwo{p}{m}.
\end{multline*}
Extract the coefficient of $t^n/n!$ on both sides to get
\begin{multline*}
P_n(x)=\sum_{p=0}^n\binom{n}{p}(-1)^mx^{m+1}m!\stirtwo{p}{m}=
\sum_{m=0}^n(-1)^m x^{m+1}m!\sum_{p=m}^n \binom{n}{p}\stirtwo{p}{m}\\
=\sum_{m=0}^n(-1)^mx^{m+1}m!\stirtwo{n+1}{m+1},
\end{multline*}
where we used the identity \cite[Eq. (6.15)]{GKP} in the last step.
The conclusion then follows.
\end{proof}

Henceforth $c_{r,j}$ denotes the coefficients of the derivative polynomials 
discussed above.
\begin{lem}
\label{LemmaTriangular}
Define $\gamma_{s+1,r}$ as the solutions of the triangular linear system 
of equations
\[
\left(
\begin{matrix}
c_{1,1}&c_{2,1}&\dots&c_{s+1,1}\\
0&c_{2,2}&\dots&c_{s+1,2}\\
\vdots&&\ddots&\vdots\\
0&0&\dots&c_{s+1,s+1}
\end{matrix}
\right)
\cdot
\left(
\begin{matrix}
\gamma_{s+1,1}\\
\gamma_{s+1,2}\\
\vdots\\
\gamma_{s+1,s+1}
\end{matrix}
\right)
=
\left(
\begin{matrix}
0\\
0\\
\vdots\\
1\\
\end{matrix}
\right).
\]
Then, $\gamma_{s+1,r}$ is given by 
\[
\gamma_{s+1,r}= \frac{(-1)^{r-1}\stirone{s+1}{r}}{s!}=(-1)^{r-1}\zt_s(\{1\}_{r-1}),
\]
where $\stirone{s+1}{r}$ denote the signless Stirling numbers of the first kind. Furthermore,
\[
\frac{2}{(1+e^{-t})^{s+1}}= \sum_{k\ge 0} \frac{t^k}{k!}\sum_{j=1}^{s+1}(-1)^{k+j-1}\gamma_{s+1,j} E_{k+j-1}(0).
\]
\end{lem}

\begin{proof}
The system of linear equations can be expressed as 
\[
\sum_{r=j}^{s+1}c_{r,j}\gamma_{s+1,r}=\delta_{j,s+1},\quad 1\le j\le s+1,
\]
where $\delta_{j,s+1}$ denotes the Kronecker delta. More explicitly, 
\[
(-1)^{j-1}(j-1)!\sum_{r=j}^{s+1}\stirtwo{r}{j}\gamma_{s+1,r}=\delta_{j,s+1}.
\]
By the inversion relationships between Stirling numbers we directly observe 
that 
\[
\gamma_{s+1,r}=\frac{(-1)^{r-1}\stirone{s+1}{r}}{s!}.
\]
By Lemma~\ref{stirmzv} we obtain the second expression.
\end{proof}

\begin{remark}
The generalized Euler polynomials $E_n^{(r)}(x)$, $r\in\N$, are defined by the 
generating function
\[
\mathcal{E}_r(t,x)=\left(\frac{2}{1+e^t}\right)^r e^{xt}=\sum_{k\ge 0}E_k^{(r)}(x)\frac{t^k}{k!},
\]
see~\cite{OLBC}. The result above implies the formula
\[
E_k^{(r)}(0)=2^{r-1}\sum_{j=1}^{r}(-1)^{j-1}\gamma_{r,j} E_{k+j-1}(0),
\]
also leading to a new formula for $E_k^{(r)}(x)$.  Cf. 
\[
E_k^{(r)}(0)=\frac{2^{r-1}}{(r-1)!}\sum_{j=0}^r s(r,j)(-1)^{r+j}E_{k+j-1}(0)
\]
which follows from \cite{MV} and gives an alternative derivation of the $\gamma_{r,j}$.
\end{remark}

\begin{proof}
By our previous result
\[
\frac{1}{(1+e^{-t})^{s+1}} = \sum_{j=1}^{s+1}\gamma_{s+1,j} \frac{d^{j-1}}{dt^{j-1}}f(t),
\]
where $f(t)=\frac12\mathcal{E}(-t,0)=\frac1{1+e^{-t}}$. 
Then
\begin{align*}
\frac{2}{(1+e^{-t})^{s+1}} &= \sum_{j=1}^{s+1}\gamma_{s+1,j}\frac{d^{j-1}}{dt^{j-1}}2f(t)\\
&=\sum_{j=1}^{s+1}\gamma_{s+1,j}\frac{d^{j-1}}{dt^{j-1}} \sum_{k\ge 0}(-1)^k E_k(0)\frac{t^k}{k!}\\
&=\sum_{j=1}^{s+1}\gamma_{s+1,j} \sum_{k\ge 0}(-1)^{k+j-1} E_{k+j-1}(0)\frac{t^k}{k!}.
\end{align*}
\end{proof}
Lemma \ref{LemmaTriangular} implies that
\[
\int_0^\infty \frac2{(1+e^{-t})^{s+1}}e^{-nmt}dt
=\sum_{j=1}^{s+1}\gamma_{s+1,j}\sum_{k\ge 0}(-1)^{k+j-1} E_{k+j-1}(0)\frac{1}{m^{k+1}n^{k+1}}.
\]
Furthermore
\begin{align*}
\mathbb{E}(Z_n^s)&=\frac{2(1-\frac{1}{2^{s+1}})}{s+1}
+\sum_{k=1}^\infty \frac{s^k}{n^k}\sum_{m=1}^\infty(-1)^{m-k}
\frac{\zt_{m-1}(\{1\}_{k-1})}{m}\\
&\quad\times\bigg(-\frac{1}{2^s(s+1)} + \frac{1}{s+1}\sum_{j=1}^{s+1}\gamma_{s+1,j}\sum_{\ell\ge 0}(-1)^{\ell+j-1} E_{\ell+j-1}(0)\frac{1}{m^{\ell}n^{\ell}}\bigg).
\end{align*}
Setting $t=0$ in Lemma \ref{LemmaTriangular} we get
\[
\frac{1}{2^s}= \sum_{j=1}^{s+1}(-1)^{j-1}\gamma_{s+1,j} E_{j-1}(0).
\]
Consequently, the first summand cancels and we get
\begin{align*}
\mathbb{E}(Z_n^s)&=\frac{2(1-\frac{1}{2^{s+1}})}{s+1}+\\
\sum_{k=1}^\infty \frac{s^k}{n^k}&\sum_{m=1}^\infty(-1)^{m-k}
\frac{\zt_{m-1}(\{1\}_{k-1})}{m(s+1)}
\sum_{j=1}^{s+1}\gamma_{s+1,j}\sum_{\ell\ge 1}(-1)^{\ell+j-1} 
E_{\ell+j-1}(0)\frac{1}{m^{\ell}n^{\ell}}\\
&=\frac{2(1-\frac{1}{2^{s+1}})}{s+1}+\\
\sum_{p=2}^\infty&\frac{(-1)^{p}}{n^p}\sum_{k=1}^{p-1}\frac{s^{k}}{s+1}\sum_{j=1}^{s+1}\gamma_{s+1,j}(-1)^{j-1} E_{p-k+j-1}(0)\zt(\overline{p-k+1},\{1\}_{k-1})
\end{align*}
by changing the order of summation.

\smallskip

Concerning arbitrary positive real $s>0$ we have to proceed in a slightly different way.  
Let $B_{n,k}(x_1,\dots,x_{n-k+1})$ denote the $k$th Bell polynomial defined by
\begin{equation}
\label{BellPolynomials}
B_{n,k}(x_1,\dots,x_{n-k+1})=\sum_{\substack{\sum_{\ell=1}^{n-k+1}j_\ell =k\\ \sum_{\ell=1}^{n-k+1}\ell j_\ell =n}}\frac{n!}{j_1!\cdots j_{n-k+1}!}\left(\frac{x_1}{1!}\right)^{j_1}\dots \left(\frac{x_{n-k+1}}{(n-k+1)!}\right)^{j_{n-k+1}}.
\end{equation}
We have
\begin{multline*}
\frac{2}{(1+e^{-t})^{s+1}}=\big(\mathcal{E}(-t,0)\big)^{s+1}
=\big(1+(\mathcal{E}(-t,0)-1))^{s+1}=\\
\sum_{j\ge 0}\frac{\fallfak{(s+1)}{j}}{j!}(\mathcal{E}(-t,0)-1)^j
=\sum_{j \ge 0}\frac{\sum_{\ell=1}^{j}\fallfak{(s+1)}{\ell}B_{j,\ell}(E_1(0),\dots,E_{j-\ell+1}(0))}{j!}(-1)^{j}t^{j}.
\end{multline*}
Consequently, 
\[
\int_0^\infty \frac{2}{(1+e^{-t})^{s+1}}e^{-mnt}dt
=\sum_{j \ge 0}(-1)^{j}\frac{\sum_{\ell=1}^{j}\fallfak{(s+1)}{\ell}B_{j,\ell}(E_1(0),\dots,E_{j-\ell+1}(0))}{(mn)^{j+1}}.
\]
Finally, 
\begin{align*}
\mathbb{E}(Z_n^s)&=\frac{2(1-\frac{1}{2^{s+1}})}{s+1}+\sum_{k=1}^\infty \frac{s^k}{n^k}\sum_{m=1}(-1)^{m-k}
\frac{\zt_{m-1}(\{1\}_{k-1})}{m}\\
&\quad\times \frac{1}{s+1}\sum_{j \ge 1}(-1)^{j}\frac{\sum_{\ell=1}^{j}\fallfak{(s+1)}{\ell}B_{j,k}(E_1(0),\dots,E_{j-\ell+1}(0))}{(mn)^{j}}\\
&=\frac{2(1-\frac{1}{2^{s+1}})}{s+1}+\sum_{p=2}^\infty \frac{(-1)^p}{n^p}\sum_{k=1}^{p-1}\frac{s^k}{s+1}\zt(\overline{p+1-k},\{1\}_{k-1})\\
&\quad\times \sum_{\ell=1}^{p-k}\fallfak{(s+1)}{\ell}B_{p-k,\ell}(E_1(0),\dots,E_{p-k-\ell+1}(0)).
\end{align*}

\smallskip

\begin{proof}[Proof of Corollary~\ref{CorollProb1}]
We use the general version of the Berry-Esseen inequality~\cite{FS2009}:
\[
\sup_{x\in \R}| F(x)-G(x)| \le c_1\int_{-T}^{T}\left|\frac{\phi_F(t)-\phi_G(t)}{t}\right|dt
+c_2\sup_{x\in\R}\big(G(x+\frac1T)-G(x)\big).
\]
From our moment expansion
\[
\E(Z_n^s)=\E(Z_\infty^s)+\mathcal{O}\Big(\frac{s\zeta(2)}{2^{s+2} n^2}\Big),
\]
we obtain for the characteristic functions $\phi_n(t)=\E(e^{i t Z_n})$ and $\phi_\infty(t)=\E(e^{i t Z_\infty})$
\[
\frac{|\phi_n(t)-\phi_\infty(t)|}{|t|}\le \frac{C_1}{n^2}.
\]
Choosing $T=n$ this gives a $\frac1n$ bound for the integral. We get $\sup_{x\in\R}\big(G(x+\frac1T)-G(x))\big)\le \frac{C_2}n$ leading to the stated result.
\end{proof}

\begin{proof}[Proof of Corollary~\ref{CorollProb2}]
By the Markov inequality we have
\[
\mathbb{P}\{|Z_n-Z_\infty|>\frac1\ell \}\le \ell^2\E((Z_n-Z_\infty)^2)=
\ell^2\big(\E(Z_n^2)+\E(Z_\infty^2)-2\E(Z_n Z_\infty)\big).
\]
The random variables $Z_n$ and $Z_\infty$ are defined in terms of the same uniform distribution and we readily obtain the expansion of 
\[
\E(Z_n Z_\infty)=\int_0^{1}(x^n+(1-x)^n)^{\frac1n}\cdot\max\{x,1-x\}dx=\frac23\cdot\frac78 +\mathcal{O}(\frac1{n^2})
\]
leading to $\mathbb{P}\{|Z_n-Z_\infty|>\frac1\ell \}\le C\cdot\frac{\ell^2}{n^2}$.
Let 
\[
E_{n,\ell}=\Big\{\omega\in \Omega:\ \big|Z_n-Z_\infty\big|>\frac{1}{\ell} \Big\},\quad n\in\N,\quad  \ell>0.
\]
We have
\[
\sum_{n\ge 1}\mathbb{P}\{E_{n,\ell}\} \le \sum_{n\ge 1} \frac{C\ell^2}{n^2} <\infty.
\]
Let $E_{\ell}=\limsup E_{n,\ell}$. By the Borel-Cantelli Lemma we have $\mathbb{P}(E_{\ell})=0$ for any $\ell>0$, giving the stated result.
\end{proof}

%

\subsection{Independent uniformly distributed random variables}
Let $U_j$ denote mutually independent standard uniformly distributed random variables, $1\le j\le r$ with $r\ge 2$, .
Further, let $\mathbf{U}$ denote the random vector
\[
\mathbf{U}=(U_1,\dots,U_r).
\]
Let $Z_n$ be defined as
\[
Z_n=\Vert \mathbf{U}\Vert_n=(U_1^n+U_2^n+\dots+U_r^n)^{\frac1n}
\]
A folklore result states that any order statistic for uniform distributions is Beta-distributed. 
In particular,
\[
Z_\infty= \Vert \mathbf{U}\Vert_{\infty}= B(r,1).
\]
We are interested in the asymptotics of $Z_n$ as $n\to\infty$ and derive asymptotics of the moments
\[
I_s=\E(Z_n^s)=\int_{[0,1]^r}(x_1^n+ \dots +x_r^n)^{\frac{s}n}d(x_1,\dots,x_r).
\]
The special case $r=2$, $Z_n=(U_1^{n}+U_2^n)^{\frac1n}$ is the direct counterpart of our earlier results for $(U^n+(1-U)^n)^{\frac1n}$.
Our asymptotic series involves for $r\ge 2$ multiple zeta values. Interestingly, for $r\ge 3$ 
variants of multiple zeta values and Euler sums appear. 
Let $\zts_r(i_1,\dots,i_k)$ denote the truncated multiple zeta star value
\begin{equation}
\label{DefStarZeta}
\zts_r(i_1,\dots,i_k)=\sum_{r\ge n_1\ge n_2\ge \dots\ge n_k\ge 1}
\frac1{n_1^{i_1}n_2^{i_2}\cdots n_k^{i_k}},
\end{equation}
and $\zts_r(i_1,\dots,i_k;x_1,\dots,x_k)$ denote the truncated weighted multiple 
zeta star value
\begin{equation}
\label{DefStarZetaWeight}
\zts_r(i_1,\dots,i_k;x_1,\dots,x_k)=\sum_{r\ge n_1\ge n_2\ge \dots\ge n_k\ge 1}
\frac{x_1^{n_1}\dots x_k^{n_k}}{n_1^{i_1}n_2^{i_2}\cdots n_k^{i_k}} .
\end{equation}
Then $\zts_r(i_1,\dots,i_k;\{1\}_k)$ is the ordinary zeta value 
$\zts_r(i_1,\dots,i_k)$, and
\[
\zts_r(\{1\}_k;\{1\}_{k-1},2)
=\sum_{r\ge n_1\ge n_2\ge \dots\ge n_k\ge 1}
\frac{2^{n_k}}{n_1n_2\cdots n_k}
=\sum_{n_1=1}^{r}\frac1{n_1}\sum_{n_2=1}^{n_1}\frac1{n_2}\dots\sum_{n_{k}=1}^{n_{k-1}}\frac{2^{n_k}}{n_k}.
\]

\begin{thm}
The random variable $Z_n=\Vert \mathbf{U}\Vert_n$ converges to $Z_\infty=B(r,1)$ with convergence 
of all positive integer moments. 
\[
\E(Z_n^s)=\frac{r}{r-1+s}\Big(1-\frac{s(r-1)}{n^2}\zt(\bar{2})+\mathcal{O}(\frac1{n^3})\Big).
\]
In particular, for $r=2$ and $Z_n=(U_1^n+U_2^n)^{\frac1n}$ we have the exact representation
\[
\E(Z_n^s)=\frac{2}{1+s}\Big(1+\sum_{p\ge 2}\frac{(-1)^p}{n^p}\sum_{\ell=0}^{p-2}s^{p-\ell-1}(-\zt(\overline{\ell+2},\{1\}_{p-\ell-2})) \Big).
\]
For $r=3$ we have the exact representation
{\footnotesize
\begin{align*}
&\E(Z_n^s)=\frac{3}{2+s}\Big(1+\sum_{p\ge 2}\frac{2(-1)^p}{n^p}\sum_{\ell=0}^{p-2}s^{p-\ell-1}(-\zt(\overline{\ell+2},\{1\}_{p-\ell-2})) \Big)+\frac{3}{2+s}\sum_{k=1}^{\infty}\frac{s^k}{n^k}\sum_{\ell_1,\ell_2\ge 0}\frac{(-1)^{\ell_1+\ell_2}}{n^{\ell_1+\ell_2+2}} \\
&\times\bigg[\sum_{i=1}^{\ell_1+1}\binom{i+\ell_2-1}{\ell_2}\sum_{m=1}^{\infty}\frac{(-1)^{m-k}\zt_{m-1}(\{1\}_{k-1})\Big(\zts_{m}(\{1\}_{\ell_1+2-i};\{1\}_{\ell_1+1-i},2)-\zts_{m}(\{1\}_{\ell_1+2-i})-\frac{1}{m^{\ell_1+2-i}}\Big)}{m^{1+i+\ell_2}} \\
&\quad +\sum_{i=1}^{\ell_2+1}\binom{i+\ell_1-1}{\ell_1}\sum_{m=1}^{\infty}\frac{(-1)^{m-k}\zt_{m-1}(\{1\}_{k-1})\Big(\zts_{m}(\{1\}_{\ell_2+2-i};\{1\}_{\ell_2+1-i},2)-\zts_{m}(\{1\}_{\ell_2+2-i})-\frac{1}{m^{\ell_2+2-i}}\Big)}{m^{1+i+\ell_1}}
\bigg].
\end{align*}
}
\end{thm}
\subsection{Exact representations}
First, we decompose the hypercube into $r$ parts according to the maximum of the $x_i$:
\[
[0,1]^r =\bigcup_{i=1}^{r}\{x_i \in[0,1],\quad 0\le x_j\le x_i,\ j\in\{1,\dots,r\}\setminus\{i\}\}.
\]
These parts are not disjoint, but their intersection is of measure zero.
By the symmetry of $I_s$ we get
\[
I_s= r \int_0^{1} \bigg(\int_{[0,x_r]^{r-1}} (x_1^n+ x_2^n+\dots +x_r^n)^{\frac{s}n}d(x_1,\dots,x_{r-1})\bigg)dx_r.
\]
We use the substitution $x_j=x_r u_j$, $d x_j=x_r du_j$ to obtain 
\[
I_s= r \int_{0}^1x_r^{r-1}\bigg(\int_{[0,1]^{r-1}} (x_r^nu_1^n+x_r^nu_2^n+ \dots +x_{r-1}^n u_{r-1}^n+1)^{\frac{s}n}d\bo{u}\bigg)dx_r.
\]
This implies that the integrals can be separated:
\begin{align*}
I_s&= r \int_{0}^1x_r^{r-1+s}d x_r \cdot \int_{[0,1]^{r-1}} (1+u_1^n+ \dots +u_{r-1}^n)^{\frac{s}n}d\bo{u}\\
&= \frac{r}{r-1+s}\cdot \int_{[0,1]^{r-1}} (1+u_1^n+ \dots +u_{r-1}^n)^{\frac{s}n}d\bo{u}.
\end{align*}
\par
In order to derive an asymptotic expansion of the remaining integral we use 
the $\exp-\log$ representation:
\[
(1+u_1^n+ \dots +u_{r-1}^n)^{\frac{s}n}=
\exp\big(\frac{s}n\ln(1+u_1^n+ \dots +u_{r-1}^n)\big)
=1+\sum_{k=1}^{\infty}\frac{s^k}{n^k k!} \ln^k(1+u_1^n+ \dots +u_{r-1}^n).
\]
Using Eq. (\ref{logp}), this implies
\[
I_s=\frac{r}{r-1+s}\cdot \bigg(1+\sum_{k=1}^{\infty}\int_{[0,1]^{r-1}}\sum_{m=1}^{\infty}\frac{s^k s(m,k)}{n^k m!} (u_1^n+ \dots +u_{r-1}^n)^m d\bo{u} \bigg),
\]
where (as above) $s(m,k)$ denotes the signed Stirling numbers of the 
first kind.
Then using Lemma \ref{stirmzv}, we have
\[
I_s=\frac{r}{r-1+s}\cdot \bigg(1+\sum_{k=1}^{\infty}\frac{s^k(-1)^k}{n^k}\sum_{m=1}^{\infty}\frac{(-1)^{m}\zt_{m-1}(\{1\}_{k-1})}{m}\int_{[0,1]^{r-1}} (u_1^n+ \dots +u_{r-1}^n)^m d\bo{u} \bigg).
\]
In order to evaluate the remaining integral we substitute $u_j=e^{-t_j}$ and obtain
\[
\int_{[0,1]^{r-1}} (u_1^n+ \dots +u_{r-1}^n)^m d\bo{u}
= \int_{[0,\infty)^{r-1}}e^{-t_1-\dots-t_{r-1}} (e^{-t_1 n}+ \dots +e^{-t_{r-1}n})^m d\bo{t}.
\]
We expand the exponentials and use the multinomial theorem. 
By the symmetry of the integrand and the fact
\[
 \int_0^{\infty}u^p e^{-ku}du=\frac{p!}{k^{p+1}}
\]
we obtain
\begin{align*}
&\int_{[0,\infty)^{r-1}}e^{-t_1-\dots-t_{r-1}} (e^{-t_1 n}+ \dots +e^{-t_{r-1}n})^m d\bo{t}\\
&=\sum_{a=1}^{r-1}\binom{r-1}{a}\sum_{\substack{j_1+\dots+j_a=m\\j_i\ge 1}}\binom{m}{j_1,\dots,j_a}
\sum_{\ell_1,\dots,\ell_a\ge 0}\frac{(-1)^{\ell_1+\dots+\ell_a}}{n^{\ell_1+\dots+\ell_a+a}j_1^{\ell_1+1}\dots j_a^{\ell_a+1}}.
\end{align*}
For $r=2$ there is only a single summand and we get
\[
\int_{[0,\infty)}e^{-t} e^{-t n m} dt 
=\sum_{\ell=0}^{\infty}\frac{(-1)^{\ell}}{(nm)^{\ell+1}}.
\]
Changing summation gives the desired result.
For $r=3$ we get
\begin{multline*}
\int_{[0,\infty)^{2}}e^{-t_1-t_2} (e^{-t_1 n}+ e^{-t_{2}n})^m d(t_1,t_2)=\\
2\sum_{\ell=0}^{\infty}\frac{(-1)^{\ell}}{(nm)^{\ell+1}}+
\sum_{\ell_1,\ell_2\ge 0}\frac{(-1)^{\ell_1+\ell_2}}{n^{\ell_1+\ell_2+2}}
\sum_{j=1}^{m-1}\binom{m}{j}\frac1{j^{\ell_1+1}(m-j)^{\ell_2+1}}.
\end{multline*}
In order to simplify the arising sums we use a classical partial fraction decomposition, which appears already in~\cite{N},
\begin{align}
\label{partialfraction}
\frac1{j^a(m-j)^b}=\sum_{i=1}^a\frac{\binom{i+b-2}{b-1}}{m^{i+b-1}j^{a+1-i}}+
\sum_{i=1}^b\frac{\binom{i+a-2}{a-1}}{m^{i+a-1}(m-j)^{b+1-i}},
\end{align}
Thus, 
\begin{align*}
&\sum_{\ell_1,\ell_2\ge 0}\frac{(-1)^{\ell_1+\ell_2}}{n^{\ell_1+\ell_2+2}}\sum_{j=1}^{m-1}\binom{m}{j}\frac1{j^{\ell_1+1}(m-j)^{\ell_2+1}}\\
&=\sum_{\ell_1,\ell_2\ge 0}\frac{(-1)^{\ell_1+\ell_2}}{n^{\ell_1+\ell_2+2}}
\bigg(\sum_{i=1}^{\ell_1+1}\frac{\binom{i+\ell_2-1}{\ell_2}}{m^{i+\ell_2}} \sum_{j=1}^{m-1}\binom{m}{j}\frac1{j^{\ell_1+2-i}}+
\sum_{i=1}^{\ell_2+1}\frac{\binom{i+\ell_1-1}{\ell_1}}{m^{i+\ell_1}}  \sum_{j=1}^{m-1}\binom{m}{j}\frac1{j^{\ell_2+2-i}}\bigg).
\end{align*}

\begin{lem}
For positive integers $r,m$ we have
\[
\sum_{j=1}^{m}\binom{m}j\frac1{j^r} =
\zts_m(\{1\}_r;\{1\}_{r-1},2) - \zts_m(\{1\}_r).
\]
\end{lem}

\begin{proof}
We use induction with respect to $r$.  For $r=1$ we have
\begin{multline*}
\sum_{j=1}^{m}\binom{m}{j}\frac{1}{j}=\int_0^1\frac{(1+t)^m-1}{t}dt
=\int_1^2\frac{t^m-1}{t-1}dt=\\
\int_1^2(t^{m-1}+t^{m-2}+\dots+t+1)dt=
\sum_{k=1}^m\frac{2^k}k-H_m=\zts_m(1;2)-\zts_m(1).
\end{multline*}
Assuming the result for $r-1$, 
\begin{multline*}
\sum_{j=1}^{m}\binom{m}j\frac1{j^r}
=\sum_{j=1}^m\sum_{k=1}^m\binom{k-1}{j-1}\frac1{j^r}
=\sum_{k=1}^m\sum_{j=1}^k\binom{k-1}{j-1}\frac1{j^r}
=\sum_{k=1}^{m}\frac1{k}\sum_{j=1}^{k}\binom{k}{j}\frac{1}{j^{r-1}}\\
=\sum_{k=1}^{m}\frac1{k}\Big(\zts_{k}(\{1\}_{r-1};\{1\}_{r-2},2) - 
\zts_k(\{1\}_{r-1})\Big)
=\zts_{m}(\{1\}_{r};\{1\}_{r-1},2)-\zts_m(\{1\}_{r}).
\end{multline*}
\end{proof}
This gives

\begin{align*}
&\sum_{\ell_1,\ell_2\ge 0}\frac{(-1)^{\ell_1+\ell_2}}{n^{\ell_1+\ell_2+2}}\sum_{j=1}^{m-1}\binom{m}{j}\frac1{j^{\ell_1+1}(m-j)^{\ell_2+1}}\\
&=\sum_{\ell_1,\ell_2\ge 0}\frac{(-1)^{\ell_1+\ell_2}}{n^{\ell_1+\ell_2+2}}
\bigg[\sum_{i=1}^{\ell_1+1}\frac{\binom{i+\ell_2-1}{\ell_2}}{m^{i+\ell_2}}\Big(\zts_{m}(\{1\}_{\ell_1+2-i};\{1\}_{\ell_1+1-i},2)-\zts_{m}(\{1\}_{\ell_1+2-i})-\frac{1}{m^{\ell_1+2-i}}\Big) \\
&\quad +\sum_{i=1}^{\ell_2+1}\frac{\binom{i+\ell_1-1}{\ell_1}}{m^{i+\ell_1}}
\Big(\zts_{m}(\{1\}_{\ell_2+2-i};\{1\}_{\ell_2+1-i},2)-\zts_{m}(\{1\}_{\ell_2+2-i})-\frac{1}{m^{\ell_2+2-i}}\Big)\bigg].
\end{align*}

\section{Outlook and Acknowledgments}
It seems that similar phenomena appear when discussing random variables $Z_n=\Vert(X_1,\dots,X_n)\Vert_n$, where the $X_i$ are i.i.d. random variables. 

\medskip

The authors thank C. Vignat for his insightful remark, pointing out the closed form expression for the $\gamma_{s+1,j}$.

\medskip

The fourth author would like to thank H. Prodinger for helping in computing an Euler sum and W.Wang for providing
a useful reference.

\end{document}